\documentclass[a4paper,11pt]{article}
\usepackage[latin1]{inputenc}
\usepackage[T1]{fontenc}

\usepackage{amsfonts}
\usepackage{amsmath,amssymb}
\usepackage{pstricks}
\usepackage{amsthm}
\usepackage[all]{xy}
\usepackage[pdftex]{graphicx}
\usepackage{graphicx}   
\newtheorem{theo}{Theorem}[section]
\newtheorem{defi}[theo]{Definition}
\newtheorem{prop}[theo]{Proposition}

\newtheorem{cor}[theo]{Corollary}
\newtheorem{lem}[theo]{Lemma}
\newtheorem{rem}[theo]{Remark}
\newtheorem{ex}{Example}
\newenvironment{dem}{\noindent
  \textit{{Proof}} : }
  {\hfill\qedsymbol\newline}
\newcommand{\IN}{\mathbb N}

\newcommand{\field}[1]{\mathbb{#1}}

\title{Jet schemes of toric surfaces}
\author{Hussein MOURTADA}

\begin{document}
\maketitle

\begin{abstract}
For $m\in \IN, m\geq 1,$ we determine the irreducible components of the $m-th$ jet scheme of a normal toric surface $S.$ We
give formulas for the number of these components and their dimensions. When $m$ varies, these components give rise to projective systems,
to which we associate a weighted graph. We prove that the data of this graph is equivalent to the data of the analytical type of $S.$
Besides, we classify these irreducible components by an integer invariant that we call index of speciality. 
We prove that for $m$ large enough, the set of components with index of speciality $1,$ is in $1-1$ correspondance with the set of
exceptional divisors that appear on the minimal resolution of $S.$
\end{abstract}

\section{Introduction}

Nash has introduced the arc space of a variety $X$ in order to investigate the intrinsic data of the various resolutions of singularities
of $X.$  The analogy with $p-$adic numbers has led Kontsevich \cite{K}, Denef and Loeser \cite{DL1} to invent motivic integration 
and to introduce several rational series that generalize analogous series in the $p-$adic context \cite{DL2}. The geometric counterpart of the theory
 of motivic integration has been used by Ein, Mustata and others  to obtain formulas controlling discrepancies in terms of invariant of jet schemes 
-these are finite dimensional approximations of the arc space-\cite{Mus2},\cite{ELM},\cite{EM},\cite{dFEI}. Roughly speaking, while we can extract 
informations about abstract resolutions of singularities  from the arc space and vice versa, we can extract 
informations about embedded resolutions of singularities from the jet schemes and vice versa. This partly explains why the arc space
of a toric variety -which has been intensively studied \cite{Mum},\cite{L},\cite{B-GS},\cite{I},\cite{IK}-  is well understood. 
Indeed, we know an equivariant abstract resolution of a toric variety, what permits to undertsand the action of the arc space of the 
torus on its arc space \cite{I}, but an equivariant embedded resolution is less accessible.  \\

Note that despite that jet schemes were the subject of numerous article in the last decade, few is known about their geometry    
for specific class of singularities, except for the following three classes: monomial ideals \cite{GS},
determinantal varieties \cite{D}, plane branches \cite{Mo1}. \\

In this article, we study the jets schemes of normal toric surface singularities. Beside being the simplest toric singularities,  
this class of singularities is interesting from the following points of view:\\

These surfaces are examples of varieties having rational singularities, but which are not necessary locally 
complete intersection, therefore we can not characterize their rationality by \cite{Mus1} via their jet schemes. We will prove that these latter have 
special properties, for example: for a given $m\in\mathbb{N},$ we will prove that the irreducible components of the $m-$th jet scheme 
of a toric surface, which have the same index of speciality (see \ref{sp} for a definition) are equidimensional. It would be interesting to figure 
out if this remains true for all rational singularities. Note that apart from the case of the $A_n$ singularities, these jet schemes are never 
irreducible, as it is the case for rational complete intersection singularities \cite{Mus1}.\\

Despite that these singularities are not complete intersections and therefore we do not have a definition of 
non-degeneration with respect to their Newton polygon in the sense of Kouchnirenko \cite{Ko},  they heuristically are 
 non-degenerate because they are desingularized with one toric morphism, so from  a jet-scheme theoretical point of view, 
they should not have vanishing components \cite{Mo1} (i.e. projective systems of irreducible components whose limit in 
the arc space are included in the arc space of the singular locus); this is what we prove in the proposition \ref{nv} and remark \ref{nnv}.
This is an approach towards defining Newton polygons without coordinates.\\

In \cite{Ni}, Nicaise has computed the Igusa motivic Poincaré series for toric surface singularities and proved that we can not extract 
the analytical type of the surface from this series. We will prove that the data of the number of irreducible components and their dimensions is equivalent to 
the data of the motivic Poincaré series. On the other hand, we will assign to the jet schemes of a toric surface a weighted graph that contains
 informations about how their irreducible components behave under the transition morphisms, and we will prove in corollary \ref{gra} that the data of this
graph is equivalent to the analytical type of the surface.\\

The Nash map for a toric surface $S$ which assigns to every irreducible component of the space of arcs centered in the singular locus an 
exceptional divisor on the minimal resolution of $S$ is bijective \cite{IK}. In general it is a difficult task to relate the irreducible components 
of the jet schemes to the irreducble components of the arc space. For a given $m,$ we classify these irreducible components by an integer invariant that
 we call index of speciality (\ref{sp}). We prove that for $m$ big enough, the components with index of speciality $1,$ are in $1-1$ correspondance with the 
exceptional divisors that appear on the minimal resolution of $S$.  This is to compare with a result that we have 
obtained in \cite{Mo2} for rational double point singularities.\\

We determine the irreducible components of the jet schemes of a toric surface as the closure of certain contact loci, and we give formulas for their number and dimensions. As a byproduct, 
we will deduce using Mustata's formula from \cite{Mus2} the log canonical threshold of the pair $S\subset \mathbb{A}^e,$
where $e$ is the embedding dimension of $S.$\\

Some of the results of this paper were announced in \cite{Mo3}.

The structure of the paper is as follows: in section two  we present a reminder on 
jet schemes and on toric surfaces. In section three we study the jet schemes of the $A_n$ singularities. 
The last section is devoted to the toric surfaces of embedding dimension bigger or equal to four.\\

AKNOWLEDGEMENTS\\

I would like to thank Monique Lejeune-Jalabert whose influence radically improved this article. 
I would like also to thank Pedro Gonzalez-Perez and Nicolas Pouyanne for discussions during the preparation of this work.\\

\section{Jet schemes and toric surfaces}
\subsection{Jet schemes}
Let $\field{K}$ be field.
Let $X$ be a $\field{K}$-scheme of finite type over $\field{K}.$ For $m \in \mathbb{N},$ the functor $F_m :
\field{K}-Schemes \longrightarrow Sets$
which to an affine scheme defined by a $\field{K}-$algebra $A$ associates
\\ $$F_m(Spec(A))=Hom_\field{K}(Spec A[t]/(t^{m+1}),X)$$
is representable by a $
\field{K}-$scheme $X_m$ \cite{V}. We call $X_m$ the $m-$th jet scheme of $X$ and we have that $F_m$ is  isomorphic to its functor of points.
In particular the $\field{K}-$points of $X_m$ are in bijection with the $
\field{K}[t]/(t^{m+1})$ points of $X$.
\\For $m,p \in \mathbb{N}, m > p$, the truncation homomorphism $A[t]/(t^{m+1}) \longrightarrow A[t]/(t^{p+1})$ induces
a canonical projection $\pi_{m,p}: X_m \longrightarrow X_p.$ These morphisms are affine and for $p<m<q$ they clearly verify
 $\pi_{m,p}\circ \pi_{q,m}=\pi_{q,p}.$
This yields an inverse system whose limit $X_{\infty}$ is a scheme called the arc space of $X.$  Note that $X_0=X$. We denote the 
canonical projections $X_m\longrightarrow X_0$ by $\pi_{m}$ and $X_{\infty}\longrightarrow X_m$ by $\Psi_m$. See \cite{EM} for more about jet schemes.\\

\begin{ex}Let $X=Spec ~\frac{\field{K}[x_0, \cdots ,x_n]}{(f_1, \cdots ,f_r)}$ be an affine $\field{K}-$scheme. For a $\field{K}$-algebra $A$, an
$A$-point of $X_m$ is a $\field{K}$-algebra homomorphism
$$\varphi : \frac{
\field{K}[x_0, \cdots ,x_n]}{(f_1, \cdots ,f_r)}\longrightarrow A[t]/(t^{m+1}).$$
This homomorphism  is completely determined by the image of $x_i,i=0, \cdots ,n$
$$x_i \longmapsto \varphi(x_i)=x_i^{(0)}+x_i^{(1)}t+  \cdots  + x_i^{(m)}t^m$$
and it should verify that  $\varphi(f_l)=f_l(\phi(x_0),  \cdots  ,\phi(x_n))\in (t^{m+1})$,  $l=1, \cdots ,r.$
\\
\\ Therefore if we set $$f_l(\phi(x_0),  \cdots  ,\phi(x_n))=\sum_{j=0}^m F_l^{(j)}(\underline{x}^{(0)}, \cdots ,\underline{x}^{(j)})~t^j~~mod ~~ (t^{m+1})$$
where $\underline{x}^{(j)}=(x^{(j)}_0, \cdots ,x^{(j)}_n),$ then we have that
$$X_m=Spec \frac{
\field{K}[\underline{x}^{(0)}, \cdots ,\underline{x}^{(m)}]}{(F_l^{(j)})_{l=1, \cdots ,r}^{ j=0, \cdots ,m}} $$
\end{ex}
\begin{ex}
From the above example, we see that the m-th jet scheme of the affine space $\mathbb{A}^n$ is
isomorphic to $\mathbb{A}^{(m+1)n}$ and that the projection $\pi_{m,{m-1}}:\mathbb{A}^n_m\longrightarrow 
\mathbb{A}^n_{m-1}$
is the map that forgets the last $n$ coordinates.
\end{ex}
\begin{rem}\label{ns}
 Note that in general, if $X$ is a non singular variety of dimension $n,$ then all the projections $\pi_{m,{m-1}}:X_m\longrightarrow X_{m-1}$ 
are locally trivial fibrations with fiber $\mathbb{A}^n.$ In particular $X_m$ is of dimension $n(m+1)$ (\cite{EM}).
\end{rem}

\subsection{Toric surfaces}
Let $S$ be a singular affine normal toric surface defined over the field $\mathbb{K}.$ There exist two coprime integers $p$ and $q$ such that $S$ is defined by the cone 
$\sigma \subset N=\mathbb{Z}^2$ generated by $(1,0)$ and $(p,q)$ and $0<p<q,$ i.e. $S=$Spec$ \mathbb{K}[x^u,~u\in \sigma^\vee \cap M]$ where 
$\sigma^\vee$ is the dual cone of $\sigma$ and $M$ is 
the dual lattice of  $N$ (\cite{O}). We have the Hirzebruch-Jung continued fraction expansion in terms
of $c_j\geq 2:$

$$\frac{q}{p}= c_2 - \cfrac{1}{c_3 -
                          \cfrac{1}{\cdots - \cfrac{1}{c_{e-1}}}}$$
which we denote by $[c_2,...,c_{e-1}].$ Let $\theta^\vee$ be the convex hull of $(\sigma^\vee \cap M) \setminus 0$ and let 
$\partial\theta^\vee$ be its boundary polygon. Let $u_1,u_2,\ldots,u_h$ be the points of $M$ lying in this order  on $\partial\theta^\vee,$
with $u_1=(0,1)$ and $u_h=(q,-p).$
Then from \cite{O}, proposition $1.21$ we have that $h=e$ is the embedding dimension of $S$ and the $u_i$ form a minimal system
 of generators of the semigroup 
$\sigma^\vee \cap M.$ For $i=1,\ldots, e,$ we will denote by $x_i$ the regular function on $S$ defined by $x^{u_i}.$ Riemenschneider has exhibited 
the generators of the ideal  defining $S$ in 
$\mathbb{A}^e=\textnormal{Spec}\field{K}[x_1,\cdots,x_e].$ 
They can be given in a quasi-determinantal format \cite{R}, \cite{St}:
\[
\left(
\begin{array}{ccccccc}
x_1 &              & x_2 & \dots &     x_{e-2} & &  x_{e-1} \\
\multicolumn{1}{c}{} &  x_2^{c_2-2} & \multicolumn{3}{c}{\dots}    &
x_{e-1}^{c_{e-1}-2} & \multicolumn{1}{c}{} \\
x_2 &              & x_3 & \dots &     x_{e-1}   &   &   x_{e} 
\end{array}
\right)
\]
where  the generalised minors of a quasi-determinant
\[\left(
\begin{array}{ccccccc}
f_1 &              & f_2 & \dots &    f_{k-1} & &  f_{k} \\
\multicolumn{1}{c}{} &  h_{1,2} & \multicolumn{3}{c}{\dots}    &
h_{k-1,k} & \multicolumn{1}{c}{} \\
g_1 &              & g_2 & \dots &    g_{k-1}   &    &   g_{k} 
\end{array}\right)
\]
are $f_ig_j - g_i(\prod_{n=i}^{j-1} h_{n,n+1})f_j$. \\

They can be written as follows:
$$E_{ij}=x_ix_j-x_{i+1}x_{i+1}^{c_{i+1}-2} x_{i+2}^{c_{i+2}-2}\cdots x_{j-2}^{c_{j-2}-2}x_{j-1}^{c_{j-1}-2}x_{j-1},$$ 
where ~~$1\leq i <j-1\leq e-1.$\\

Let $b_i \in \mathbb{N},b_i\geq 2,$ be such that $q/(q-p)=[b_1,\ldots,b_r].$
Let $l_0=(1,0),\ldots,l_{s+1}=(p,q)$ in this order be the elements of $N$ lying on the compact edges of the boundary  $\partial\theta$
of the convex hull $\theta$ of $(\sigma \cap N) \setminus 0.$ 
\begin{prop}\label{oda}
 We have that $r=s$ and is equal to the number of irreducible components of the exceptional curve for the minimal resolution of singularities 
of $S.$ Moreover we have that 
$$c_2+\cdots+c_{e-1}-2(e-2)+1=s.$$ 
\end{prop}

See lemma 1.22 and corollary 1.23 in  \cite{O} for a proof.

\section{Jet schemes of the $A_n$ singularities}

Let $S$ be the variety defined in $\mathbb{A}^3$ by the equation $f(x,y,z)=xy-z^{n+1}=0.$ $X$ has an $A_n$ singularity at the origin
 $0$ and is nonsingular elsewhere. Note that an affine toric suface of embedding dimension $3$ has this type of singularities
(see section 2.1). If we set 
$$f(\sum_{i=0}^{m}x^{(i)}t^i,\sum_{i=0}^{m}y^{(i)}t^i,\sum_{i=0}^{m}z^{(i)}t^i)=\sum_{i=0}^{i=m}F^{(i)}t^i~~~ \mbox{mod}~~t^{m+1},~~~~(\diamond)$$
then $S_m$ is defined in $\mathbb{A}^{3(m+1)}=\mathbb{A}^{3}_m$ by the ideal
$I_m=(F^{(0)},F^{(1)},...,F^{(m)}).$

By the remark \ref{ns} the morphism  $\pi_m^{-1}(S \backslash 0)\longrightarrow S \backslash 0$ is a trivial fibration,
therefore we have that $\overline{\pi_m^{-1}(S\backslash0)}$
is an irreducible component of $S_m$ of codimension $m+1$ in $\mathbb{A}^{3}_m.$ On the other hand, we will prove in the coming lines that 
the codimension of $S_m^0:=\pi_m^{-1}(0)$  in  $\mathbb{A}^{3}_m$ is $m+2,$ which means that $S_m$ is irreducible 
for every $m \in \mathbb{N}~:$ indeed, since $I_m$ is generated by $m+1$ equations, any irreducible component of $S_m$ could  have codimension
at most $m+1.$ (Note that this fact -that $S_m$ is irreducible- follows form \cite{Mus1} because $S$ is  locally a complete intersection 
with a rational singularity, but we give here a direct proof in this simple case.)\\
We claim that for $m \leq n,$ we have $S_m^0=Z_m^0$, where $Z \subset \mathbb{A}^3$ is the hypersurface defined by $xy=0.$
Indeed, a $m-$jet $\gamma_m=(x=\sum_{i=0}^{m}x^{(i)}t^i,y=\sum_{i=0}^{m}y^{(i)}t^i,z=\sum_{i=0}^{m}z^{(i)}t^i) \in (\mathbb{A}^{3})_m$ 
centered at the origin (i.e.$x^{(0)}=y^{(0)}=z^{(0)}$) is in $S_m^0$ if and only if $xy-z^{n+1} \equiv 0 ~\mbox{mod}~ t^{m+1}$, but since $z_0=0$ 
and $m\leq n,$ but we have that $ord_tz^{n+1}\geq n+1>m+1$, therefore this is equivalent to $ord_txy \geq n+1$ and therefore to $\gamma \in Z_m^0$.\\
But clearly for $m\leq  n$, the irreducible commponents of $Z_m^0=S_m^0$ are the subvarities defined by the ideals
$$I^l_m=(x^{(0)},...,x^{(l-1)},y^{(0)},...,y^{(m-l)},z^{(0)}),l=1,...,m.$$
Notice the the codimensions of $V(I^l_m)$ in $\mathbb{A}_m^3$ is equal to $m+2$ for $l=1,...,m$. We deduce that for $m\leq n, S_m$ is
 irreducible of codimension $m+1$. On the other hand, for $m\geq n+1$ we have that $\pi^{-1}_{m,n}(V(I^l_n))$ is 
defined in $(\mathbb{A}^{3})_m$ by the ideal $I_m^l=(I_n^l,J_{m-(n+1)}^l)$ where $J_{m-(n+1)}^l$ is the ideal obtained from the ideal  
 defining $X_{m-(n+1)}$ in $\mathbb{A}^{3})_{m-(n+1)}$ by changing variables. Indeed if we set 
$$f(\sum_{i=l}^{m}x^{(i)}t^i,\sum_{i=n-l+1}^{m}y^{(i)}t^i,\sum_{i=1}^{m}z^{(i)}t^i)=$$

$$f(t^l(\sum_{i=0}^{m-l}x^{(l+i)}t^i),t^{n-l+1}(\sum_{i=0}^{m-(n-l+1)}y^{(n-l+1+i)}t^i),t(\sum_{i=0}^{m-1}z^{(i+1)}t^i))
=$$
$$t^{n+1}f(\sum_{i=0}^{m-l}x^{(l+i)}t^i,\sum_{i=0}^{m-(n-l+1)}y^{(n-l+1+i)}t^i,\sum_{i=0}^{m-1}z^{(i+1)}t^i)$$

$$=t^{n+1}(\sum_{i=0}^{i=m-(n+1)}G_l^{(i)}t^i)~~~ \mbox{mod}~~t^{m+1},~~~~~~~~~~~~~~~~(\diamond\diamond)$$
then $J_{m-(n+1)}^l$ is generated by $G_l^{(i)},i=0,\ldots,m-(n+1),$ and  by comparing $(\diamond)$ with $(\diamond \diamond),$
we get that $$G_l^{(i)}=F^{(i)}(x^{(l)},\ldots,x^{(m)},y^{(n-l+1)},\ldots,y^{(m)},z^{(1)},\ldots,z^{(m)}).$$
\\
We deduce that for $l=1,...,n,$ $$Codim~~(\pi^{-1}_{m,n}(V(I^l_n)),\mathbb{A}_m^{3})=n+2 +Codim~~(S_{m-(n+1)},\mathbb{A}^{3}_{m-(n+1)}) . $$ 
This implies by a simple induction that for $l=1,...,n,$
$$Codim~~\pi^{-1}_{m,n}(V(I^l_n))=m+2.$$
Therefore  $Codim~~(S_m^0,\mathbb{A}^3_m)=m+2,$ so $S_m$ is irreducible. It follows that $\pi^{-1}_{m,n}(V(I^l_n))$ which is isomorphic to
 $S_{m-(n+1)}\times \mathbb{A}^{2n+1}$
 is irreducible and we conclude:

\begin{theo}\label{lc}
For $m \in \mathbb{N},n\geq 1,$  The scheme of $m-$th jets centered in the singular locus of an $A_n$ 
singularity is a locally complete intersection scheme.
For $m\leq n$ this scheme has $m$ irreducible components of codimension $m+2$. For $m \geq n+1,$ it has $n$ irreducible components  
each of codimension $m+2.$ 
\end{theo}

\section[Jet schemes of toric surfaces]{Jet schemes of toric surfaces of embedding dimension $e \geq 4$}
We keep the notations introduced in section 2 and we begin by introducing some more notations.
Let $f\in \field{K}[x_1,\ldots,x_e]$ ; for $m,p \in \mathbb{N}$ such that  $p\leq m,$ we set:
$$Cont^p(f)_{m}(resp.Cont^{>p}(f)_{m}):=\{\gamma \in S_m \mid ord_{\gamma }(f)=p(resp.>p)\},$$
$$Cont^p(f)=\{\gamma \in S_{\infty}\mid ord_{\gamma }(f)=p\},$$
where $ord_{\gamma }(f)$ is the $t-$order of $f\circ\gamma.$ 
\\For $a,b\in \mathbb{N},~b\not=0,$ we denote by $\lceil\frac{a}{b}\rceil$  the round-up of $\frac{a}{b}$. For $i=2,\cdots, e-1,$
$ s \in \{1,\ldots,\lceil\frac{m}{2}\rceil\}$(i.e. $m\geq 2s-1\geq1)$ and
$ l \in \{s,\ldots,m_i^s\}, $ where $$m_i^s:=min\{(c_i-1)s, (m+1)-s \},~$$
we set  $$D_{i,m}^{s,l}:=Cont^s(x_i)_{m} \cap Cont^{l}(x_{i+1})_{m},$$
and $$C_{i,m}^{s,l}:=\overline{D_{i,m}^{s,l}}.$$

If $R$ is a ring, $I\subseteq R$ an ideal and $f \in R$, we denote by $V(I)$ the subvariety of $Spec~R$ defined by $I$
and by $D(f)$ the open set $D(f):=\textnormal{Spec}~R_f.$
\begin{lem} \label{eff}
 
For $i=2,\cdots,e-1,$ $s\geq 1,$    the ideal defining $C_{i,2s-1}^{s,s}$ in $\mathbb{A}^e_{2s-1}$ is 
$$I_{i,2s-1}^{s,s }=(x_j^{(b)},1 \leq j \leq e,0 \leq b < s).$$
Note that $C_{i,2s-1}^{s,s}$ does not depend on $i$. For $j=1,e,$ we set $$C_{j,2s-1}^{s,s}:=C_{i,2s-1}^{s,s},~i=2,\cdots,e-1.$$  
\end{lem}

\begin{dem}
Let us prove that  $D_{i,2s-1}^{s,s}=V(I_{i,2s-1}^{s,s})\cap D(x_i^{(s)}x_{i+1}^{(s)}).$
Let $\gamma \in \mathbb{A}^e_{2s-1}$ such that $ord_{\gamma}x_i=ord_{\gamma}x_{i+1}=s.$
So, we have $ord_{\gamma}x_i^{c_i}=c_is> 2s-1$ because $c_i\geq2.$ If moreover $\gamma$ lies in $S_{2s-1}$, then it satisfies $E_{i-1,i+1}$ mod $t^{2s},$ which is equivalent to $ord_{\gamma}x_{i-1} \geq s $, because $x_i^{c_i}\circ \gamma \equiv0$ mod $t^{2s}$
and $ord_{\gamma}x_{i+1}=s$. 
The same argument, using $E_{i-2,i},E_{i,i+2}$ and so on by induction, using the other $E_{ji}$'s and $E_{ij}$'s, gives that $ord_{\gamma}x_{j}\geq s.$
We deduce $$D_{i,2s-1}^{s,s}\subset V(I_{i,2s -1}^{s,s})\cap D(x_i^{(s)}x_{i+1}^{(s)}).$$
The opposite inclusion comes from the fact that a jet in $V(I_{i,2s-1}^{s,s})\cap D(x_i^{(s)}x_{i+1}^{(s)})
\subset \mathbb{A}^e_{2s}$
satisfies all the equations of $S$ modulo $t^{2s}.$ Since $V(I_{i,2s-1}^{s,s})\subset \mathbb{A}^e_{2s-1}$  is irreducible, the lemma follows.
\end{dem}
\begin{lem}\label{inc}
For $i=2,\cdots,e-1,~m \in \mathbb{N},~  s \in \{1,\ldots,\lceil\frac{m}{2}\rceil\}$ and $l \in \{s,\ldots,m_i^s \}, $ we have that
$$C_{i,m}^{s,l}\subset \pi_{m,2s-1}^{-1}(C^{s,s}_{i,2s-1}).$$
\end{lem}
\begin{dem}
For $\gamma \in D_{i,m}^{s,l},$ we have that 
$$ E_{i-1,i+1}\circ \gamma=x_{i-1}\circ \gamma x_{i+1}\circ \gamma -(x_{i}\circ \gamma)^{c_i} \in (t)^{m+1}. $$
If $c_is\geq m+1,$ then $ord_{\gamma}x_{i-1}\geq m+1-l\geq s,$ and if $c_is<m+1$ then $ord_{\gamma}x_{i-1}=c_is-l\geq s.$ Moreover, since for $i<j-1\leq e-1(resp.~ 1\leq j<i-1),$ we have 
$$E_{ij}\circ \gamma =x_i \circ \gamma x_j \circ \gamma -x_{i+1}\circ \gamma x_{i+1}^{c_{i+1}-2}\circ \gamma \cdots x_{j-1}^{c_{j-1}-2}\circ \gamma x_{j-1}\circ \gamma \in (t)^{m+1},$$
$$(resp.~ E_{ji}\circ \gamma =x_j\circ \gamma x_i \circ \gamma -x_{j+1}\circ \gamma x_{j+1}^{c_{j+1}-2}\circ \gamma \cdots x_{i-1}^{c_{i-1}-2}\circ \gamma x_{i-1}\circ \gamma \in (t)^{m+1}, )$$  
$$ord_{\gamma}x_{i}=s,ord_{\gamma}x_{i+1}\geq s(resp.~ord_{\gamma}x_{i-1}\geq s),  $$
$$ c_{i+1}, (resp.~c_{i-1})\geq 2~~ \textnormal{and}~~ m+1\geq 2s.$$
We get by ascending (resp. descending) induction on $j$ that $ord_{\gamma}x_{j} \geq s,$ and therefore 
$D_{i,m}^{s,l}\subset \pi_{m,2s-1}^{-1}(C^{s,s}_{i,2s-1}).$ The lemma follows since $\pi_{m,2s-1}^{-1}(C^{s,s}_{i,2s-1})$ is closed.

\end{dem}

 \begin{lem}\label{syz} For $m\geq 2s-1\geq 1,~i=1,\ldots, e-1,$
  $$\pi_{m,2s-1}^{-1}(C_{i,2s-1}^{s,s}\cap D(x_{i}^{(s)}))=\{\gamma \in \mathbb{A}^e_m~;~ord_{\gamma}x_{j}\geq s,~ j=1,\cdots,e,~ord_{\gamma}x_{i}=s,$$
$$~ord_{\gamma}E_{i-1,i+1}\geq m+1,~ ord_{\gamma}E_{j,i}(resp.ord_{\gamma}E_{i,j})\geq m+1,~\textnormal{for}~ 1\leq j <i-1$$ $$(resp.~i<j-1\leq e-1)  \}.$$
 \end{lem}
\begin{dem}The inclusion $``\subset''$ is an immediate consequence of lemma \ref{eff}. To get the other inclusion, it is enough to check that for 
every $\gamma \in \mathbb{A}^e_m$ enjoying the conditions listed above, we also have 
$ord_{\gamma}E_{jh}\geq m+1$ for $1\leq j<h-1\leq e-1.$\\
If $i<j,$ the syzygie 
$$
 x_iE_{jh}-x_jE_{ih}+x_{j+1}^{c_{j+1}-2}\cdots x_{h-1}^{c_{h-1}-2}x_{h-1}E_{i,j+1}=0 ~~~~~~~~~~~~~~~~~~~~~~~~~~(4.1)
$$

implies that $ord_{\gamma}E_{jh}\geq m+1,$ because $ord_{\gamma}x_j$ and $ord_{\gamma}x_{h-1}\geq s$ and
$ord_{\gamma}x_i=s.$\\
Similarly if $h<i,$ the syzygie 
$$x_iE_{jh}-x_hE_{ji}+x_{j+1}x_{j+1}^{c_{j+1}-2}\cdots x_{h-1}^{c_{h-1}-2}E_{h-1,i}=0~~~~~~~~~~~~~~~~~~~~~~~~~~~(4.2)$$
implies that $ord_{\gamma}E_{jh}\geq m+1,$ because $ord_{\gamma}x_h$ and $ord_{\gamma}x_{j+1}\geq s$ and
$ord_{\gamma}x_i=s.$\\
Assume now that $1\leq j<i-1$ and $h=i+1;$ the syzygie

$$x_{i+1}E_{ji}-x_iE_{j,i+1}+x_{j+1}x_{j+1}^{c_{j+1}-2}\cdots x_{i-1}^{c_{i-1}-2}E_{i-1,i+1}=0~~~~~~~~~~~~~~~~~~~(4.3)$$
implies that $ord_{\gamma}E_{j,i+1}\geq m+1.$  \\
Similarly if $j=i-1$ and  $i+1<h\leq e,$ the syzygie 
$$x_{i-1}E_{ih}-x_iE_{i-1,h}+x_{i+1}^{c_{i+1}-2}\cdots x_{h-1}^{c_{h-1}-2}x_{h-1}E_{i-1,i+1}=0~~~~~~~~~~~~~~~~~(4.4)$$
implies that $ord_{\gamma}E_{i-1,h}\geq m+1.$ \\
Finally , if $1\leq j< i-1$ and $i+1<h\leq e,$ the syzygie 
$$x_jE_{ih}-x_iE_{jh}+x_{i+1}^{c_{i+1}-2}\cdots x_{h-1}^{c_{h-1}-2}x_{h-1}E_{j,i+1}=0~~~~~~~~~~~~~~~~~~~~~~~~~(4.5)$$ implies that
$ord_{\gamma}E_{j,h}\geq m+1,$ taking into account that we have shown above that $ord_{\gamma}E_{j,i+1}\geq m+1.$\\
 
\end{dem}

\begin{prop}\label{irr}
 
 For $i=2,\cdots,e-1,~m \in \mathbb{N},~  s \in \{1,\ldots,\lceil\frac{m}{2}\rceil\}$ and $l \in \{s,\ldots,m_i^s \}, $
$C_{i,m}^{s,l}$ is irreducible, and its codimension in  $\mathbb{A}^e_m$ is equal to $$ se+(m-(2s-1))(e-2).$$  
\end{prop} 
\begin{dem}
First, since the ideal defining $S$ in $\mathbb{A}^e$ is generated by $E_{jh},~1\leq j<h-1\leq e-1,$ we have that
$$D_{i,m}^{s,l}\subset U_{i,m}^{s,l}:=\{\gamma \in \mathbb{A}^e_m;~ord_{\gamma}E_{ij}(\textnormal{resp.}~ord_{\gamma}E_{ji})\geq m+1~
\textnormal{for}~i<j-1\leq e-1$$
$$(\textnormal{resp.}~1\leq j<i-1),~ 
ord_{\gamma}E_{i-1,i+1}\geq m+1,~ord_{\gamma}x_i=s,~ord_{\gamma}x_{i+1}=l\}.$$
For $\gamma \in U_{i,m}^{s,l},$ we have by the proof of \ref{inc} that for $j=1,\cdots,e,~ord_{\gamma}x_{j} \geq s.$ It follows from lemma
\ref{syz} that $D_{i,m}^{s,l}=U_{i,m}^{s,l}.$

The irreducibility of $C_{i,m}^{s,l}$ follows from the fact that $D_{i,m}^{s,l}=U_{i,m}^{s,l}$ is isomorphic to the product 
of a two dimensional torus by an affine space. Indeed, set $x_j \circ \gamma=\sum_{0\leq \nu\leq m}x_j^{(\nu)}t^{(\nu)},$ 
$~1\leq j\leq e.$
If $ord_{\gamma}x_i=s$ and $ord_{\gamma}x_{i+1}=l,$ we have $ord_{\gamma}E_{i-1,i+1}\geq m+1,$  if and only if $x_{i-1}^{(\nu)}=0$ 
for $0\leq \nu \leq m-l$ if $c_is \geq m+1$ (resp. $x_{i-1}^{(\nu)}=0$ for $0\leq \nu \leq c_is-l$ and is a polynomial function
 of $x_i^{(s)},\cdots,x_i^{(m-c_is+l)},1/x^{(l)}_{i+1},x_{(i+1)}^{(l)},\cdots x_{(i+1)}^{(m-c_is+l)}$ for $c_is-l\leq \nu \leq m-l if c_is<m+1).$
Similarly , $ord_{\gamma}E_{ij}$(resp. $ord_{\gamma}E_{ji})\geq m+1$ for $i+1\leq j \leq e$ (resp. $1\leq j<i-1)$ if and only 
if $x_j^{(\nu)}=0$ for $0\leq \nu<s$ and is a polynomial function of $1/x_i^{(s)},x_i^{(s)},\cdots,x_i^{(m-s)},x_{i+1}^{(l)},\cdots,
x_{i+1}^{(m-l)}$ for $s\leq \nu \leq m-s$(resp. $x_j^{(\nu)}=0$ for $0\leq \nu <s$ and is a polynomial function of 
$1/x^{(s)}_{i},x^{(s)}_{i},\cdots,x^{(m-s)}_{i},x^{(s)}_{i-1},\cdots,x^{(m-s)}_{i-1}$ for $s\leq \nu\leq m-s$ since 
$ord_{\gamma}x_{i-1}\geq s$ as soon as $ord_{\gamma}E_{i-1,i+1}\geq m+1).$ As a consequence, the codimension of $D^{s,l}_{i,m},$
hence of its closure $C^{s,l}_{i,m},$  is $$m+s+1+(e-i-1)(m-s+1)+(i-2)(m-s+1)=$$
$$(e-2)(m+1)-(e-4)s=se+(m-(2s-1))(e-2).$$

\end{dem}
\begin{prop}\label{nv} 
\begin{enumerate}
 \item For $i=2,\cdots, e-1$ and $m,s\in \mathbb{N}$ such that $m\geq 2s-1$ and $l \in \{s,\ldots,m_i^s\},$ we have $\Psi^{-1}_m(D_{i,m}^{s,l})\not= 
\emptyset.$
\item For $s\in \mathbb{N},~s\geq 1,~ Cont^s(x_1)\cap Cont^s(x_2)\not=\emptyset.$
\end{enumerate}
\end{prop}
\begin{dem} 
(1)-We will prove that there exists an arc $h$ on $S$, whose generic point lies in the torus, and such that $h \in Cont^s(x_i) \cap Cont^l(x_{i+1}).$ Note 
that the data of such an arc $h$ on $S$ is equivalent to the data of a vector $v_h=(a,b) \in \sigma \cap N;$ moreover 
$\forall u \in M\cap \sigma^\vee$, we have that $h \in Cont^{v_h.u}(x^u),$ where we denote by $v_h.u$ the scalar product of 
$v_h$ and $u$, and by $x^u$ the regular function defined by $u$ on $S$. Let $u_i, i=1,\cdots ,e,$ be the system of minimal generators of 
$\sigma^\vee \cap M,$ defined in 2.2 such that $x^{u_i}=x_i.$ Therefore to prove that there exists an arc $h$ as above, 
it is sufficient to prove that there exists $(a,b) \in \sigma \cap N$ such that $(a,b).u_i=s$ and $(a,b).u_{i+1}=l$ where   
$x^{u_i}=x_i,$ and $ x_{i+1}=x^{u_{i+1}}.$ Since $u_{i}$ and $u_{i+1}$ determine a $\mathbb{Z}-$basis of $M,$ there exists a unic $(a,b) \in N$ such that 
$(a,b).u_i=s$ and $(a,b).u_{i+1}=l.$ Let's prove that $(a,b)$ is in the interior of $\sigma,$ i.e. that for $j=1,\cdots,e,~ (a,b).u_j >0.$ 
Since $u_{i-1}=c_iu_i-u_{i+1},$ we have that $(a,b).u_{i-1}=c_is-l$ which is greater than or equal to $s$ because by hypothesis we have 
$s\leq l \leq s(c_i-1).$ Similarly we have that $(a,b).u_{i+2}=c_{i+1}l-s$ which is greater than or equal to $l.$ Since $c_i\geq 2,$ for $i=1,\cdots,e,$ by 
descending (repectively ascending) induction we find that $(a,b).u_{j-1}\geq (a,b).u_{j},$ for $j=2,\cdots,i$ (respectively  
 $(a,b).u_{j-1} \leq (a,b).u_{j},$ for $j=i+2,\cdots,e)$ and the proposition follows. \\
(2)-We have that $u_1=(0,1), u_2=(1,0).$ we need to prove that the unic vector $v=(a,b) \in N$ such that $(a,b).(0,1)=b=s$ and   
$(a,b).(1,0)=a=s,$ belongs also to $\sigma,$ and this is clear.  
 
\end{dem}
\begin{lem}\label{ac}
For $i=2,\ldots,e-1,$ let $X^i=$Spec$\mathbb{K}[x_{i-1},x_i,x_{i+1}]/(x_{i-1}x_{i+1}-x_i^{c_i}).$ For $m\geq 2s,$ let

$$V_{i,m}^s:=\{\gamma \in X^i_m, ord_{\gamma}(x_j)\geq s,~j=i-1,i+1,~ord_{\gamma}(x_i)=s\}$$
and for $l\in\{s,\ldots,m_i^s\},$ let 
$$\Delta_{i,m}^{s,l}:=\{\gamma \in X^i_m, ord_{\gamma}(x_i)= s,~ord_{\gamma}(x_{i+1})=l\}.$$
Then, the irreducible components of $\overline{V_{i,m}^s}$ are the $\overline{\Delta_{i,m}^s},~l\in\{s,\ldots,m_i^s\}.$ 

\end{lem}
\begin{dem}
First, assume that $m+1\leq c_is,$ so that $m_i^s=m+1-s.$ We have that $$V_{i,m}^s= \{\gamma \in \mathbb{A}^3_m~;~ord_{\gamma}x_j\geq s, j=i-1,i+1,
ord_{\gamma}x_i=s $$
$$\textnormal{and} ~~ord_{\gamma}x_{i-1}+ord_{\gamma}x_{i+1}\geq m+1\} $$
and for $l\in\{s,\ldots,m+1-s\},$
$$\Delta_{i,m}^{s,l}=\{\gamma \in \mathbb{A}^3_m~;~ord_{\gamma}x_i= s, ord_{\gamma}x_{i+1}=l,ord_{\gamma}x_{i-1}\geq m+1-l\}=$$

$$V(x_{i-1}^{(0)},\ldots,x_{i-1}^{(m-l)},x_i^{(0)},\ldots,x_i^{(s-1)},x_{i+1}^{(0)},\ldots,x_{i+1}^{(l-1)})\cap D(x^{(s)}_{i}x_{i+1}^{(l)}). $$
Since $s\leq l\leq m+1-s,$ we have that  $\Delta_{i,m}^{s,l}\subset V_{i,m}^s,$ so $\cup_{s\leq l\leq m+1-s}\overline{\Delta_{i,m}^{s,l}}
\subset \overline{V_{i,m}^s}.$ Now for $\gamma \in V_{i,m}^s,$ we have that $ord_{\gamma}x_i=s, l:=ord_{\gamma}x_{i+1} \geq s$ and
$ord_{\gamma}x_{i-1}\geq m+1-l.$ If $l\leq m+1-s,$ we thus have that $\gamma \in \Delta_{i,m}^{s,l};$ if $l>m+1-s,$ we have that 
$ord_{\gamma}x_{i-1}\geq s,$ hence $\gamma \in V(x_{i-1}^{(0)},\ldots,x_{i-1}^{(s-1)},x_i^{(0)},\ldots,x_i^{(s-1)},x_{i+1}^{(0)},
\ldots,x_{i+1}^{(m-s)})=\overline{\Delta_{i,m}^{s,m+1-s}},$ hence the claim.  \\
Now assume that $c_is<m+1,$ so that $m_i^s=(c_i-1)s.$ For $l\in\{s,\ldots,(c_i-1)s\}$ and $\gamma  \in \Delta_{i,m}^{s,l},$ we thus have that 
$ord_{\gamma}x_{i}=s,ord_{\gamma}x_{i+1}=l\geq s,$ and $ord_{\gamma}x_{i-1}+l=c_is,$ hence 
$ord_{\gamma}x_{i-1}=c_is-l\geq s,$ therefore $\Delta_{i,m}^{s,l} \subset V_{i,m}^s$ and $\cup_{s\leq l\leq m+1-s}\overline{\Delta_{i,m}^{s,l}}
\subset \overline{V_{i,m}^s}.$ \\ On the other hand  $V_{i,m}^s=(\pi^i_{m,c_is-1})^{-1}$ where
$\pi^i_{m,c_is-1}:X_m^i\longrightarrow X_{c_is-1}^i$ is the natural map and for $s\leq l \leq (c_i-1)s,
\Delta_{i,m}^{s,l}=(\pi^i_{m,c_is-1})^{-1}(\Delta_{i,c_is-1}^{s,l}).$ Now we have just seen that 
$\overline{V_{i,c_is-1}^{s}}=\cup_{s\leq l \leq (c_i-1)s}\overline{\Delta_{i,c_is-1}^{s,l}}$
and that \\
 $\overline{\Delta_{i,c_is-1}^{s,l}}=V(x_{i-1}^{(0)},\ldots,x_{i-1}^{(c_is-l-1)},x_i^{(0)},\ldots,x_i^{(s-1)},x_{i+1}^{(0)},\ldots,x_{i+1}^{(l-1)}).$

As a consequence $(\pi^i_{m,c_is-1})^{-1}(\overline{\Delta_{i,c_is-1}^{s,l}})$ is isomorphic to a product of an affine space by the 
space of $(m-c_is)-$jets of the surface  Spec$\mathbb{K}[x^{(c_is-l)}_{i-1},x^{(s)}_i,x^{(l)}_{i+1}]/({x_{i-1}}^{(c_is-l)}{x_{i+1}}^{(l)}-
{x_i^{(s)}}^{c_i}),$
and this latter is irreducible by section $3,$ hence coincides with $\overline{\Delta_{i,c_is-1}^{s,l}}.$ So 
$\overline{V_{i,m}^{s}} \subset \cup_{s\leq l \leq (c_i-1)s}\overline{\Delta_{i,c_is-1}^{s,l}},$ hence the claim.

\end{dem}

\begin{prop}\label{gs}
 Let $m,s \in \mathbb{N}$ such that $m\geq 2s-1.$
\begin{enumerate}
\item For $i=2,\cdots,e-1,$ the irreducible components of $\overline{\pi_{m,2s-1}^{-1}(C_{i,2s-1}^{s,s}\cap D(x_{i}^{(s)}))}$ 
are the $C_{i,m}^{s,l}, l\in \{s,\cdots,m_i^s\}.$ 
\item  For $i=1,e,$ we have that $\pi_{m,2s-1}^{-1}(C_{i,2s-1}^{s,s}\cap D(x_{i}^{(s)}))$ is irreducible of codimension 

$$se+(m-(2s-1))(e-2)$$
in $\mathbb{A}_m^e.$

\end{enumerate}
\end{prop}
 
\begin{dem}
(1) By the lemmas \ref{inc} and \ref{syz}, we have that $D^{s,l}_{i,m}\subset \pi_{m,2s-1}^{-1}(C_{i,2s-1}^{s,s}\cap D(x_{i}^{(s)}))=$

$$\{\gamma \in \mathbb{A}^e_m~;~ord_{\gamma}x_{j}\geq s,~ j=1,\cdots,e,~ord_{\gamma}x_{i}=s,
~ord_{\gamma}E_{i-1,i+1}\geq m+1,$$
$$ord_{\gamma}E_{j,i}(resp.ord_{\gamma}E_{i,j})\geq m+1,~\textnormal{for} 1\leq j <i-1(resp.~i<j-1\leq e-1)  \}. $$
The projection $\mathbb{A}^e\longrightarrow \mathbb{A}^3$ which sends $(x_1,\ldots,x_e)$ to
$(x_{i-1},x_i,x_{i+1})$ induces a natural map $p^i:S\longrightarrow X^i$ and the induced map $p_m^i:S_m\longrightarrow X_m^i$ sends 
$\pi_{m,2s-1}^{-1}(C_{i,2s-1}^{s,s}\cap D(x_{i}^{(s)}))$ (resp. $D^{s,l}_{i,m})$ into 
$V_{i,m}^s:=\{\gamma \in X^i_m, ord_{\gamma}(x_j)\geq s,~j=i-1,i+1,~ord_{\gamma}(x_i)=s\}$ (resp.
$\Delta_{i,m}^{s,l}:=\{\gamma \in X^i_m, ord_{\gamma}(x_i)= s,~ord_{\gamma}(x_{i+1})=l\}).$ Now in view of lemma \ref{syz},
the maps
 $$\pi_{m,2s-1}^{-1}(C_{i,2s-1}^{s,s}\cap D(x_{i}^{(s)}))\longrightarrow V_{i,m}^s ~~\textnormal{and}~~D^{s,l}_{i,m} \longrightarrow \Delta_{i,m}^{s,l}$$
are isomorphic to a trivial fibration of rank $s(e-3).$ By lemma \ref{ac}, the irreducible components of $\overline{V_{i,m}^s}$ 
are the $\overline{\Delta_{i,m}^{s,l}},~l\in\{s,\ldots,m_i^s\}.$ Since $V_{i,m}^s=\overline{V_{i,m}^s}\cap D(x_i^{(s)}),$
we thus have $V^s_{i,m}=\cup_l(\overline{\Delta_{i,m}^{s,l}}\cap D(x_i^{(s)}));$  so 
$\pi_{m,2s-1}^{-1}(C_{i,2s-1}^{s,s}\cap D(x_{i}^{(s)}))\simeq \cup_l\Omega_{i,m}^{s,l}$ where $\Omega_{i,m}^{s,l}=
(\overline{\Delta_{i,m}^{s,l}}\cap D(x_i^{(s)}))\times \mathbb{A}^{s(e-3)}.$ As a consequence  $\Omega_{i,m}^{s,l}$ is irreducible 
and we have that $D_{i,m}^{s,l}\subset \Omega_{i,m}^{s,l}.$ Moreover $$Codim(\Omega_{i,m}^{s,l},\mathbb{A}_m^{e})=
(e-3)(m+1)+(m+s+1)-s(e-3)=$$
$$(m+1)(e-2)-s(e-4)=Codim(C_{i,m}^{s,l},\mathbb{A}_m^{e});$$ hence  $C_{i,m}^{s,l}=\overline{\Omega_{i,m}^{s,l}}$ 
and the claim follows since $C_{i,m}^{s,l}\not=C_{i,m}^{s,l'}$ for $l\not=l'.$\\

(2) Assume $i=1,$ the case $i=e$ follows in the same way. We first check that 
$$ \pi_{m,2s-1}^{-1}(C_{i,2s-1}^{s,s}\cap D(x_{1}^{(s)}))=$$
$$\{\gamma \in \mathbb{A}^e_m, ord_{\gamma}(x_j)\geq s,~j=1,\ldots,e,~ord_{\gamma}(x_1)=s,$$
$$ord_{\gamma}E_{1j}\geq m+1 ~~\textnormal{for} ~3\leq j\leq e\} .$$
The inclusion $``\subset``$ is clear. To get the opposite inclusion we have to prove that the conditions just listed imply that 
$ord_{\gamma}E_{jh}\geq m+1$ for $2\leq j<h-1\leq e-1.$ This is an immediate consequence of the syzygie 
$$x_1E_{jh}-x_jE_{1h}+x_{j+1}^{c_{j+1}-2}\cdots x_{h-1}^{c_{h-1}-2}x_{h-1}E_{1,j+1}=0.$$
Therefore, $\pi_{m,2s-1}^{-1}(C_{i,2s-1}^{s,s}\cap D(x_{1}^{(s)}))$ is isomorphic to the product of $\mathbb{K}^*$ by an affine space of dimension 
$(m-s)+(m-s+1)+s(e-2)$ and its Zariski closure is irreducible of ccodimension $(m+1)(e-2)-s(e-4)$ in $\mathbb{A}^e_m.$

\end{dem}

\begin{lem}\label{id}
For $i=2,\ldots,e-2,$ we have that $$C_{i,m}^{s,s}=C_{i+1,m}^{s,m_{i+1}^s}.$$
 \end{lem}
\begin{dem}
If $m+1\leq c_{i+1}s,$ by definition $m^s_{i+1}=m+1-s,$ and in view of lemma \ref{eff} and lemma \ref{inc}, we have that 
$D_{i,m}^{s,s} \subset \pi_{m,2s-1}^{-1}(C_{i+1,2s-1}^{s,s}\cap D(x_{i+1}^{(s)})).$ Now by proposition \ref{gs}, the irreducible
components of  $\overline{\pi_{m,2s-1}^{-1}(C_{i+1,2s-1}^{s,s}\cap D(x_{i+1}^{(s)}))}$ are the $C_{i+1,m}^{s,l}$ for 
$l\in\{s,\ldots,m_{i+1}^s\}.$ Since $C_{i,m}^{s,s}=\overline{D_{i,m}^{s,s}}$ is irreducible, and its codimension in 
$\mathbb{A}^e_m$ coincides with the codimension of any of the $C_{i+1,m}^{s,l},$ there exists $l$ such that 
$C_{i,m}^{s,s}=C_{i+1,m}^{s,l}$ whith $ s\leq l \leq m+1-s.$ So $D^{s,s}_{i,m}$ and $D^{s,l}_{i+1,m}$ are dense open subsets 
of $C_{i,m}^{s,s}$ and there exists $\gamma \in D^{s,s}_{i,m} \cap D^{s,l}_{i+1,m}.$ We thus have 
$ord_{\gamma}x_i=ord_{\gamma}x_{i+1}=s,$ and $ord_{\gamma}x_{i+2}=l.$ But 
$E_{i,i+2}=x_ix_{i+2}-x_{i+1}^{c_{i+1}}$ and $ord_{\gamma}E_{i,i+2} \geq m+1.$
Since $m+1\leq c_{i+1}s,$ this implies $ord_{\gamma} x_{i+2}=l\geq m+1-s,$ so $l=m+1-s,$ i.e. $C_{i,m}^{s,s}=C_{i+1,m}^{s,m_{i+1}^s}.$ \\
Assume now that $m+1 > c_{i+1}s;$ for any $\gamma \in  D^{s,s}_{i,m},$ we have that $ord_{\gamma}x_i=ord_{\gamma} x_{i+1}=s$ and   
$ord_{\gamma} E_{i,i+2}\geq m+1,$ hence $ord_{\gamma} x_{i+2}=(c_{i+1}-1)s=m_{i+1}^s$ which implies 
that $D^{s,s}_{i,m}\subset D^{s,m_{i+1}^s}_{i+1,m}.$  Since both are irreducible and have the same dimension, we deduce by passing to the closure that  
$C_{i,m}^{s,s}=C_{i+1,m}^{s,m_{i+1}^s}.$ 
\end{dem}

Let $S_m^0:=\pi_{m}^{-1}(O)$, where $O$ is the singular point of $S.$ Note that $\overline{\pi_{m}^{-1}(S-\{0\})}$
is an irreducible component of $S_m$ of codimension $(m+1)(e-2)$ in $\mathbb{A}^e_m;$  we will see that the irreducible components of $S_m^0$ have codimension less than or equal to $(m+1)(e-2),$ therefore they are irreducible components of $S_m.$
\begin{prop}\label{cov}
   $$S_m^0=\bigcup_{i\in \{2,...,e-1\},s \in \{1,\ldots,\lceil\frac{m}{2}\rceil\},l \in \{s,\ldots,m_i^s\}}  C_{i,m}^{s,l}.$$
\end{prop}
\begin{dem}
We first look at \textbf{the case m=2n+1}, $~n\geq 0.$ We claim that $$S_{2n+1}^0 =\bigcup_{i\in \{1,...,e\},s \in \{1,\ldots,n\}} \pi_{2n+1,2s-1}^{-1}(C_{i,2s-1}^{s,s}\cap D(x_{i}^{(s)}))\cup C_{i,2n+1}^{n+1,n+1}.~~~~~~~~(\diamond)$$
The proof of the claim is by induction on $n.$ 
By lemma \ref{eff}, we have that $S^0_1=C_{i,1}^{1,1}$ for any $i=1,...,e,$ hence the case $n=0.$
Using the inductive hypothesis for $n-1,$ and the fact that for  $s \in \{1,\ldots,n-1\}$ we have that $\pi_{2n-1,2s-1}\circ \pi_{2n+1,2n-1}=\pi_{2n+1,2s-1},$ we obtain:  
$$S_{2n+1}^0=\pi_{2n+1,2n-1}^{-1}(S_{2n-1}^0)=$$
$$\bigcup_{i\in \{1,...,e\},s \in \{1,\ldots,n-1\}} \pi_{2n+1,2s-1}^{-1}(C_{i,2s-1}^{s,s}\cap D(x_{i}^{(s)}))
\cup \pi_{2n+1,2n-1}^{-1}( C_{i,2n-1}^{n,n}).$$
The claim follows from the stratification 

$C_{i,2n-1}^{n,n}=\bigcup_{j=1,\cdots,e}(C_{i,2n-1}^{n,n}\cap D(x_{j}^{(n)}))\cup (C_{i,2n-1}^{n,n}\cap V(x_1^{(n)},
\cdots,x_e^{(n)})), $ \\
and from the fact that by lemma \ref{eff} $\pi_{2n+1,2n-1}^{-1}(C_{i,2n-1}^{n,n}\cap V(x_1^{(n)},
\cdots,x_e^{(n)})) =C_{i,2n+1}^{n+1,n+1}.$\\
We then conclude the proof of the proposition for $m=2n+1$ in two steps : First by using proposition \ref{gs} (2). Second, by proposition \ref{nv} we have that
 hence $Cont^s(x_1)\cap Cont^s(x_2)\not=\emptyset,$ 
that $\pi_{2n+1,2s-1}^{-1}(C_{i,2s-1}^{s,s}\cap D(x_{2}^{(s)})) \cap  \pi_{2n+1,2s-1}^{-1} (C_{i,2s-1}^{s,s}\cap D(x_{1}^{(s)}))\not=\emptyset~;$ 
since by \ref{gs} (1) this latter is irreducible, its generic point coincides with the generic point 
of one of the irreducible compnents of $\overline{\pi_{2n+1,2s-1}^{-1}(C_{i,2s-1}^{s,s}\cap D(x_{1}^{(s)}))}.$  \\

\textbf{The case m =2(n+1)}, $n\geq0$ : by ($\diamond$) we just need to prove that for $n\geq 0,$ and $i=1,\ldots,e$ we have that  
$$\pi_{2(n+1),2n+1}^{-1}(C_{i,2n+1}^{n+1,n+1})=
\cup_{\{i=2,\cdots,e-1~;~l=n+1,\cdots, (2(n+1))_i^{n+1}\}}C_{i,2(n+1)}^{n+1,l}.$$

First note that by lemma \ref{eff} and lemma \ref{inc} we have the inclusion
$$\pi_{2(n+1),2n+1}^{-1}(C_{i,2n+1}^{n+1,n+1}) \supset \cup_{\{i=2,\cdots,e-1~;~l=n+1,\cdots, (2(n+1))_i^{n+1}\}}C_{i,2(n+1)}^{n+1,l}.~~~(\diamond)$$

 The proof of the opposite inclusion is by induction on the embedding dimension $e$ of $S.$\\
 First assume that $e=4;$ the equations defining $S$ in $\mathbb{A}^4$ are $E_{13},E_{14},E_{24}.$ So the ideal defining 
$\pi_{2(n+1),2n+1}^{-1}(C_{i,2n+1}^{n+1,n+1})$ in $\mathbb{A}_{2(n+1)}^4$ is generated by
$$(x_j^{(0)},\ldots,x_j^{(n)},E^{(2n+2)}_{13},E^{(2n+2)}_{14},E^{(2n+2)}_{24};~j=1,\ldots,4),$$
hence every irreducible component of $\pi_{2(n+1),2n+1}^{-1}(C_{i,2n+1}^{n+1,n+1})$ has codimension  in $\mathbb{A}_{2(n+1)}^4$ less than or equal 
to $4(n+1)+3=4n+7.$\\
Now we have that $$\pi_{2(n+1),2n+1}^{-1}(C_{i,2n+1}^{n+1,n+1})=
\bigcup_{j=1,\ldots,4}\pi_{2(n+1),2n+1}^{-1}((C_{i,2n+1}^{n+1,n+1}\cap D(x_j^{(n+1)})))$$
$$\cup~~ \pi_{2(n+1),2n+1}^{-1}( (C_{i,2n+1}^{n+1,n+1}\cap V(x_1^{(n+1)},\ldots,x_4^{(n+1)})))$$
$$=\bigcup_{j=1,\ldots,4}\overline{\pi_{2(n+1),2n+1}^{-1}((C_{i,2n+1}^{n+1,n+1}\cap D(x_j^{(n+1)})))}$$ 
$$\cup ~~\pi_{2(n+1),2n+1}^{-1}( (C_{i,2n+1}^{n+1,n+1}\cap V(x_1^{(n+1)},
\ldots,x_4^{(n+1)}))).$$

Moreover since by proposition \ref{nv} we have that $Cont^{n+1}(x_1)\cap Cont^{n+1}(x_2)\not=\emptyset,$ we deduce that

  $$\pi_{2(n+1),2n+1}^{-1}(C_{i,2n+1}^{n+1,n+1}\cap D(x_{1}^{(n+1)})) \cap  \pi_{2(n+1),2n+1}^{-1} (C_{i,2n+1}^{n+1,n+1}
\cap D(x_2^{(n+1)})))\not=\emptyset.$$
By proposition \ref{gs} (2), $\overline{\pi_{2(n+1),2n+1}^{-1}(C_{i,2n+1}^{n+1,n+1}\cap D(x_{1}^{(n+1)}))}$  is irreducible, therefore it coincides 
with an irreducible components of $\overline{\pi_{2(n+1),2n+1}^{-1} (C_{i,2n+1}^{n+1,n+1}
\cap D(x_2^{(n+1)})))}.$\\ Similarly $\overline{\pi_{2(n+1),2n+1}^{-1}(C_{i,2n+1}^{n+1,n+1}\cap D(x_{4}^{(n+1)}))}$
coincides  with an irreducible components of $\overline{\pi_{2(n+1),2n+1}^{-1} (C_{i,2n+1}^{n+1,n+1}
\cap D(x_3^{(n+1)})))}.$ \\
In addition by lemma \ref{eff} and  proposition \ref{gs}. 1), we have that for $j=2,3,$ 
$$\overline{\pi_{2(n+1),2n+1}^{-1} (C_{i,2n+1}^{n+1,n+1}\cap D(x_j^{(n+1)}))}=
\bigcup_{l=1,\ldots,(2(n+1))_j^{n+1}}C_{j,2(n+1)}^{n+1,l}.$$ 
Hence $\pi_{2(n+1),2n+1}^{-1} (C_{i,2n+1}^{n+1,n+1})=$
$$\bigcup_{l=1,\ldots,(2(n+1))_j^{n+1};~j=2,3}C_{j,2n+1}^{n+1,l}~~\cup~~\pi_{2(n+1),2n+1}^{-1}( (C_{i,2n+1}^{n+1,n+1}\cap V(x_1^{(n+1)},
\ldots,x_4^{(n+1)}))).$$

Finally  we have that $\pi_{2(n+1),2n+1}^{-1}( (C_{i,2n+1}^{n+1,n+1}\cap V(x_1^{(n+1)},
\ldots,x_4^{(n+1)})))=$ $$\{\gamma \in S_{2(n+1)},~ord_\gamma x_j\geq n+2,~j=1,\ldots,4\}
=\{\gamma \in \mathbb{A}^4_{2(n+1)},~ord_\gamma x_j\geq n+2,~j=1,\ldots,4\}$$
$$=V(x_j^{(0)},\ldots,x_j^{(n+1)};~j=1,\ldots,4)$$
is irreducible of codimension $4(n+2)$ in $\mathbb{A}^4_{2(n+1)}$. Since $4(n+2)>4n+7,$ it is not an irreducible component of 
$\pi_{2(n+1),2n+1}^{-1}(C_{i,2n+1}^{n+1,n+1}),$ hence the claim.\\
We now assume the lemma to be true for toric surfaces $\tilde{S}$ of embedding dimension $\tilde{e}$ with $4\leq \tilde{e}
\leq e-1.$ We have that
$\pi_{2(n+1),2n+1}^{-1}(C_{i,2n+1}^{n+1,n+1})=$
$$\pi_{2(n+1),2n+1}^{-1}(C_{i,2n+1}^{n+1,n+1}\cap D(x_e^{(n+1)})) ~\cup~\pi_{2(n+1),2n+1}^{-1}(C_{i,2n+1}^{n+1,n+1}\cap V(x_e^{(n+1)}).$$
Again by proposition \ref{nv} and proposition \ref{gs}, $\overline{\pi_{2(n+1),2n+1}^{-1}(C_{i,2n+1}^{n+1,n+1}\cap D(x_e^{(n+1)}))}$ 
coincides with one of the irreducible components of $\overline{\pi_{2(n+1),2n+1}^{-1}(C_{i,2n+1}^{n+1,n+1}\cap D(x_{e-1}^{(n+1)}))},$
namely the $C_{e-1,2(n+1)}^{n+1,l}$ for $l\in \{n+1,\ldots,(2(n+1))^{n+1}_{e-1}\}.$

So it remains to determine $\pi_{2(n+1),2n+1}^{-1}(C_{i,2n+1}^{n+1,n+1}\cap V(x_e^{(n+1)}).$ The discussion splits into two cases:\\
\textbf{i)} There exists $h\in \{3,\ldots,e\}$ such that $c_{h-1}>2$ and $c_h=\cdots=c_{e-1}=2.$\\
By lemma \ref{eff}, we have that $\pi_{2(n+1),2n+1}^{-1}(C_{i,2n+1}^{n+1,n+1}\cap V(x_e^{(n+1)})=$
$$\{\gamma \in S_{(2n+1)};~ord_\gamma x_j\geq n+1,~1\leq j \leq e-1,~ord_\gamma x_e\geq n+2\}=$$
$$\{\gamma \in \mathbb{A}^e_{(2n+1)};~ord_\gamma x_j\geq n+1,~1\leq j \leq e-1,~ord_\gamma x_e\geq n+2, $$
$$ord_\gamma E_{jk}\geq 2n+3,~1\leq j<k-1\leq e-1\}.$$
Now recall that $E_{e-2,e}=x_{e-2}x_e-x_{e-1}^{c_{e-1}}.$ If $h<e,$ we have that $c_{e-1}=2,$ so for $\gamma \in \mathbb{A}^e_{2(n+1)}$ such that
$ord_\gamma x_{e-2}\geq n+1,ord_\gamma x_{e}\geq n+2$ and $ord_\gamma E_{e-2,e}\geq 2n+3,$ we thus have that $2ord_\gamma x_{e-1}\geq 2n+3$
hence $ord_\gamma x_{e-1}\geq n+2.$ Similarly, if $i\geq h,$ for $\gamma \in \mathbb{A}^e_{2(n+1)}$
such that
$ord_\gamma x_{i-1}\geq n+1,ord_\gamma x_{i+1}\geq n+2$ and $ord_\gamma E_{i-1,i+1}\geq 2n+3,$ we get that $ord_\gamma x_{i}\geq n+2.$
By descending induction on $i,$ this shows that 
$$\pi_{2(n+1),2n+1}^{-1}(C_{i,2n+1}^{n+1,n+1}\cap V(x_e^{(n+1)})\subset V(x_h^{(n+1)},\ldots,x_e^{(n+1)}).$$
Note that this inclusion is verified by definition when $h=e.$ Moreover,
for $\gamma \in  \mathbb{A}^e_{2(n+1)}$ 
such that $ord_\gamma x_j\geq n+1 (\textnormal{resp.}~n+2)$ for $1\leq j<h(\textnormal{resp.} ~h\leq j \leq e),$ we have that 
$ord_\gamma E_{jk}\geq 2n+3$ if $h\leq k\leq e,$ indeed we have that
$$ord_\gamma x_jx_k\geq n+1 +n+2=2n+3,~~~\textnormal{and}$$
$$ ord_\gamma x_{j+1}x_{j+1}^{c_{j+1}-2}\ldots x^{c_{k-1}-2}x_{k-1}\geq 3(n+1)~(\textnormal{resp.}~n+1+n+2)$$
for $k=h(\textnormal{resp.}~k>h).$ Therefore we have that
$$\pi_{2(n+1),2n+1}^{-1}(C_{i,2n+1}^{n+1,n+1}\cap V(x_e^{(n+1)})=\{\gamma \in  \mathbb{A}^e_{2(n+1)};
~ord_\gamma x_j\geq n+1,$$
$$ 1\leq j\leq h-1,
ord_\gamma x_j\geq n+2,h\leq j \leq e,$$
$$ord_\gamma E_{jk}\geq 2n+3, 1\leq j<k-1\leq h-2\}.~~~~(\diamond\diamond)$$ 
If $h\geq 5,$ this can be interpreted geometrically as follows: Let $\tilde{S}$ be the toric surface in $\mathbb{A}^{h-1}=
\textnormal{Spec}\mathbb[x_1,\ldots,x_{h-1}]$ defined by the ideal generated by $(E_{jk}, 1\leq j<k-1\leq h-2)$ and for
$i=2,\ldots,h-2,$ $m\in \mathbb{N}, s\in\{1,\ldots,\lceil\frac{m}{2}\rceil\},l\in\{s,\ldots,m_i^s\}$ let 
$$\tilde{D}^{s,l}_{i,m}=\{\gamma \in \tilde{S}_m;~ord_\gamma x_i=s,ord_\gamma x_{i+1}=s\}$$
and  $\tilde{C}^{s,l}_{i,m}=\overline{\tilde{D}^{s,l}_{i,m}};$ finally for $m>p,$ let $\tilde{\pi}_{m,p}:\tilde{S}_m\longrightarrow \tilde{S}_p$ be the 
canonical projection. By lemma \ref{eff} again, we have that 
$$\tilde{\pi}_{2(n+1),2n+1}^{-1}(\tilde{C}_{i,2n+1}^{n+1,n+1})=
\{\gamma \in  \mathbb{A}^{h-1}_{2(n+1)};~ord_\gamma x_j\geq n+1,1\leq j\leq h-1,$$
$$ord_{\gamma}E_{jk}\geq 2n+3, 1\leq j<k-1\leq h-2\}.$$

Therefore we deduce that $\pi_{2(n+1),2n+1}^{-1}(C_{i,2n+1}^{n+1,n+1}\cap V(x_e^{(n+1)})=$
$$\tilde{\pi}_{2(n+1),2n+1}^{-1}(\tilde{C}_{i,2n+1}^{n+1,n+1})\times
\textnormal{Spec}\mathbb{K}[x^{(n+2)}_j,\ldots,x_j^{(2(n+1))},j=h,\ldots,e],$$
which by the inductive hypothesis equal to 
$$\bigcup_{i=2,\ldots,h-2;~l=n+1,\ldots,(2(n+1))_i^{n+1}} 
\tilde{C}_{i,2(n+1)}^{n+1,l}\times \textnormal{Spec}\mathbb{K}[x^{(n+2)}_j,\ldots,x_j^{(2(n+1))},j=h,\ldots,e].$$
Newt we claim that 
$$\bigcup_{i=2,\ldots,h-2;~l=n+1,\ldots,(2(n+1))_i^{n+1}} 
C_{i,2(n+1)}^{n+1,l}\subset V(x_h^{(n+1)},\ldots,x_e^{(n+1)}).$$
Indeed, let $\gamma \in D_{i,2(n+1)}^{n+1,l}$ for some $i$ and $l$ in the above union. We have that $ord_\gamma x_j\geq n+1$ for
$1\leq j\leq e,$ $ord_\gamma x_i=n+1$ and $ord_\gamma E_{ie}\geq 2n+3.$ Since $i\leq h-2$ and $c_{h-1}>2,$ this implies that   
$$ord_\gamma x_{i+1}x_{i+1}^{c_{i+1}-2}\ldots x^{c_{e-1}-2}x_{e-1}\geq 2n+3,$$
therefore  $ord_\gamma x_ix_e\geq 2n+3,$ thus $ord_\gamma x_e\geq n+2,$ and
$\gamma \in \pi_{2(n+1),2n+1}^{-1}(C_{i,2n+1}^{n+1,n+1})$ and since we have proved that 
$$\pi_{2(n+1),2n+1}^{-1}(C_{i,2n+1}^{n+1,n+1}\cap V(x_e^{(n+1)})\subset  V(x_h^{(n+1)},\ldots,x_e^{(n+1)}),$$
we deduce that  $C_{i,2(n+1)}^{n+1,l}=\overline{ D_{i,2(n+1)}^{n+1,l}}\subset  V(x_h^{(n+1)},\ldots,x_e^{(n+1)}).$
Finally by proposition \ref{irr}, $C_{i,2(n+1)}^{n+1,l}(\textnormal{resp.}~\tilde{C}_{i,2(n+1)}^{n+1,l})$ is irreducible 
of codimension $(n+1)e+e-2 (\textnormal{resp.}~(n+1)(h-1)+h-3)$ in $\mathbb{A}_{2(n+1)}^e(\textnormal{resp.}~\mathbb{A}_{2(n+1)}^{h-1}),$
therefore $$\dim~ C_{i,2(n+1)}^{n+1,l}=\dim~ \tilde{C}_{i',2(n+1)}^{n+1,l'}\times \textnormal{Spec}\mathbb{K}[x^{(n+2)}_j,\ldots,x_j^{(2(n+1))},j=h,\ldots,e]$$
for any $i'\in\{2,\ldots h-2\},l'\in\{n+1,\ldots,(2(n+1))^{n+1}_{i'}\},$
and we deduce from the first inclusion $(\diamond)$ that $C_{i,2(n+1)}^{n+1,l}$ coincides with 
$\tilde{C}_{i',2(n+1)}^{n+1,l'}\times \textnormal{Spec}\mathbb{K}[x^{(n+2)}_j,\ldots,x_j^{(2(n+1))},j=h,\ldots,e]$ where 
$i'\in\{2,\ldots h-2\},l'\in \{n+1,\ldots,(2(n+1))^{n+1}_{i'}\}.$\\
But we have that $ord_\gamma x_i=n+1,$ $ord_\gamma (x_{i+1})=l$ for $\gamma$ the generic point of $C_{i,2(n+1)}^{n+1,l},$ therefore since 
$i+1\leq h-1,$ we have that $ord_{\tilde{\gamma}}x_i=n+1$ and $ord_{\tilde{\gamma}}x_{i+1}=l$ for   $\tilde{\gamma}$ the generic point
of $\tilde{C}_{i',2(n+1)}^{n+1,l'}.$ Therefore $\tilde{\gamma}\in \tilde{C}_{i,2(n+1)}^{n+1,l}$
and we deduce that $\tilde{C}_{i',2(n+1)}^{n+1,l'}\subset\tilde{C}_{i,2(n+1)}^{n+1,l}.$ But since they are irreducible of the same codimension 
they are equal, so we have that  
$$C_{i,2(n+1)}^{n+1,l}=\tilde{C}_{i,2(n+1)}^{n+1,l}\times \textnormal{Spec}\mathbb{K}[x^{(n+2)}_j,\ldots,x_j^{(2(n+1))},j=h,\ldots,e].$$
 We thus have that 
$$\pi_{2(n+1),2n+1}^{-1} (C_{i,2n+1}^{n+1,n+1}\cap V(x_e^{(n+1)}))=\bigcup_{i=2,\ldots,h-2; l=n+1,\ldots,(2(n+1))_i^{n+1}} C_{i,2(n+1)}^{n+1,l},$$
and the claim follows.(Note that we get that
$$\bigcup_{i=2,\ldots,h-2,e-1,l=n+1;~\ldots,(2(n+1))_i^{n+1}} C_{i,2(n+1)}^{n+1,l}=
\bigcup_{i=2,\ldots,e-1,l=n+1;~\ldots,(2(n+1))_i^{n+1}} C_{i,2(n+1)}^{n+1,l}$$
as an immediate consequence of lemma \ref{eff} and lemma \ref{id}.)\\
If $h=4,$ let $\tilde{S}$ be the toric surface in $\mathbb{A}^3=\textnormal{Spec}\field{K}[x_1,x_2,x_3]$ defined by the ideal $(E_{1,3})$ and let 
$\tilde{C}_{i,2(n+1)}^{n+1}=\{\gamma \in \tilde{S}_{2(n+1)};~ord_\gamma x_j\geq n+1,~j=1,2,3\}.$ The equality $(\diamond\diamond)$ reduces to 
$$\pi_{2(n+1),2n+1}^{-1} (C_{i,2n+1}^{n+1,n+1}\cap V(x_e^{(n+1)}))= \tilde{C}_{i,2(n+1)}^{n+1}\times
 \textnormal{Spec}\mathbb{K}[x^{(n+2)}_j,\ldots,x_j^{(2(n+1))},j=4,\ldots,e].$$
Since $E_{13}=x_1x_3-x_2^{c_2},$ if $c_2>2,$ $\tilde{C}_{i,2(n+1)}^{n+1}\subset
\textnormal{Spec}\mathbb{K}[x^{(n+2)}_j,\ldots,x_j^{(2(n+1))},j=1,\ldots,3]$
is defined by the ideal $(x_1^{(n+1)}x_3^{(n+1)}),$ so $\tilde{C}_{i,2(n+1)}^{n+1}=V(x_1^{(n+1)})\cup V(x_3^{(n+1)})$ while it is irreducible 
if $c_2=2.$ \\
We check as above that 

$$\bigcup_{~l=n+1,\ldots,(2(n+1))_2^{n+1}} 
C_{2,2(n+1)}^{n+1,l}\subset V(x_4^{(n+1)},\ldots,x_e^{(n+1)})$$
and that $\dim  C_{2,2(n+1)}^{n+1,l}$ coincides with the dimension of any irreducible components 
of $\tilde{C}_{i,2(n+1)}^{n+1}\times
 \textnormal{Spec}\mathbb{K}[x^{(n+2)}_j,\ldots,x_j^{(2(n+1))},j=4,\ldots,e].$ Again in view of $(\diamond),$ 
each  $C_{2,2(n+1)}^{n+1,l}$ is an irreducible component of $\tilde{C}_{i,2(n+1)}^{n+1}\times
 \textnormal{Spec}\mathbb{K}[x^{(n+2)}_j,\ldots,x_j^{(2(n+1))},j=4,\ldots,e].$\\
If $c_2=2,$ then $(2(n+1))_2^{n+1}=n+1$ and we thus have 
$$\pi_{2(n+1),2n+1}^{-1} (C_{i,2n+1}^{n+1,n+1}\cap V(x_e^{(n+1)}))=C_{2,2(n+1)}^{n+1,n+1}.$$

If $c_2>2,$ we have that $(2(n+1))_2^{n+1}=n+2,$ and the same argument as above shows that 
$$C_{2,2(n+1)}^{n+1,n+1}=
 V(x_1^{(n+1)}\times\textnormal{Spec}\mathbb{K}[x^{(n+2)}_j,\ldots,x_j^{(2(n+1))},j=4,\ldots,e]$$
$$C_{2,2(n+1)}^{n+1,n+2}=
 V(x_3^{(n+1)}\times\textnormal{Spec}\mathbb{K}[x^{(n+2)}_j,\ldots,x_j^{(2(n+1))},j=4,\ldots,e].$$
We thus have 
$$\pi_{2(n+1),2n+1}^{-1} (C_{i,2n+1}^{n+1,n+1}\cap V(x_e^{(n+1)}))=\bigcup_{l=n+1;~\ldots,(2(n+1))_i^{n+1}} C_{2,2(n+1)}^{n+1,l}$$
hence the claim.
\\
Finally if $h=3,$ by $(\diamond\diamond)$ we have that 
$\pi_{2(n+1),2n+1}^{-1} (C_{i,2n+1}^{n+1,n+1}\cap V(x_e^{(n+1)}))=$
$$\textnormal{Spec}\mathbb{K}[x^{(n+1)}_j,\ldots,x_j^{(2(n+1))},j=1,2]\times\textnormal{Spec}\mathbb{K}[x^{(n+2)}_j,\ldots,x_j^{(2(n+1))},j=3,\ldots,e].$$
Now we have that $C_{2,2(n+1)}^{n+1,n+1}\subset V(x_3^{(n+1)},\ldots,x_e^{(n+1)}).$
Indeed, for $\gamma \in  D_{i,2n+1}^{n+1,n+2},$ we have that $ord_\gamma x_2=n+1,ord_\gamma x_3=n+2,$ 
$ord_\gamma x_j\geq n+1,j=4,\ldots,e$ and  $ord_\gamma E_{2j}\geq 2n+3$ for $j=4,\ldots,e.$
Since $c_3=\ldots=c_{e-1}=2,$ this implies that $ord_\gamma x_j\geq n+2$ for $j=4,\ldots,e,$
so $\gamma \in V(x_3^{(n+1)},\ldots,x_e^{(n+1)}).$ We conclude that 
$\pi_{2(n+1),2n+1}^{-1} (C_{i,2n+1}^{n+1,n+1}\cap V(x_e^{(n+1)}))=C_{2,2(n+1)}^{n+1,n+2}$
because both sets are irreducible and have the same dimension, and the claim follows in this case.\\

\textbf{ii)}
If $c_2=\cdots=c_{e-1}=2$ then  
$$\pi_{2(n+1),2n+1}^{-1}(C_{i,2n+1}^{n+1,n+1})=V(x_i^{(0)},\ldots,x_i^{(n)},~i=1,\ldots,n,$$  
$$x^{(n+1)}_ix_j^{(n+1)}-x^{(n+1)}_{i-1}x_{j-1}^{(n+1)},1\leq i <j-1\leq e-1).$$
The ideal generated by $(x^{(n+1)}_ix_j^{(n+1)}-x^{(n+1)}_{i-1}x_{j-1}^{(n+1)},~1\leq i <j-1\leq e-1),$ 
is isomorphic to the ideal defining $S$ in $\mathbb{A}^e,$ hence it is prime and $\pi_{2(n+1),2n+1}^{-1}(C_{i,2n+1}^{n+1,n+1})$ is irreducible.
Since by proposition \ref{nv} we have that  $$\pi_{2(n+1),2n+1}^{-1}(C_{i,2n+1}^{n+1,n+1}\cap D(x_{e-1}^{(n+1)} \cap D(x_e^{(n+1)})\not=\emptyset,$$
then it is dense  $\pi_{2(n+1),2n+1}^{-1}(C_{i,2n+1}^{n+1,n+1}),$ and we deduce that 
$$\pi_{2(n+1),2n+1}^{-1}(C_{i,2n+1}^{n+1,n+1})=C_{e-1,2(n+1)}^{n+1,n+1},$$
thus the proposition in this case.\\

\end{dem}  
\begin{rem}Note that the argument that we use in the proposition \ref{cov} for $e=4$ does not work in general. The argument works in the case  $e=4$ because the number of equations that define $S \subset\mathbb{A}^e$ (this number is ${2 \choose e-1})$ is less or equal to $e$ if   and only if
$e\leq 4.$  
 
\end{rem}

\begin{theo} \label{th}
Let $m\in  \mathbb{N}, ~m\geq 1.$
Modulo the identifications 
$C_{i,m}^{s,s}=C_{i+1,m}^{s,m_{i+1}^s},$  the irreducible components of $S_m^0:=\pi_m^{-1}(0)$ are the $C_{i,m}^{s,l},~i=2, \cdots, e-1,$
$ s \in \{1,\ldots,\lceil\frac{m}{2}\rceil\} $ and
$ l \in \{s,\ldots,m_i^s \}\}.$ 
\end{theo} 
\begin{dem}
By proposition \ref{cov}, $S_m^{(0)}$ is covered by the $C_{i,m}^{s,l}.$ But apart from the identifications above, $C_{i,m}^{s,l}\not \subset C_{i',m}^{s,l'},$
because by proposition \ref{gs}, there exist hyperplane coodinates that contain the one but not the other, and by proposition \ref{irr} they have the same dimension. On the other hand $C_{i,m}^{s,l}\not \subset C_{i',m}^{s',l'},$ if $s <s',$ because by proposition \ref{gs} the $C_{i,m}^{s,l}$ has non-empty intersection with  $D(x_i^{(s)}),$ but $C_{i',m}^{s',l'}\subset V(x_i^{(s)})$. Finally, $C_{i',m}^{s',l'} \not \subset C_{i,m}^{s,l},$ because by proposition \ref{irr} the codimension of the first one is less than or equal to the codimension of the second one, and the theorem follows.
\end{dem}

\begin{rem}\label{nnv}
Given Theorem \ref{th}, proposition \ref{nv} means that there are no vanishing components.
 \end{rem}
We obtain a graph $\Gamma$ by representing every iireducible components of $S_m^0,m\geq 1,$ by a vertex $v_{i,m}$ and by joining the vertices 
$v_{i_1,m+1}$ and $v_{i_0,m}$ if the morhphism $\pi_{m+1,m}$ induces a morphism between the corresponding irreducible components.
To every $i=2,\ldots,i-1,$  $\Gamma$ contains a subgraph $\Gamma_i;$ the identifications  $C_{i,m}^{s,s}=C_{i+1,m}^{s,m_{i+1}^s},$
are translated by an identification between infinite lines of $\Gamma_i$ and $\Gamma_{i+1}$ indexed by the index of speciality $s.$ In the figure below,
these lines are the broken lines having the same color.

\includegraphics[width=160mm,height= 160mm]{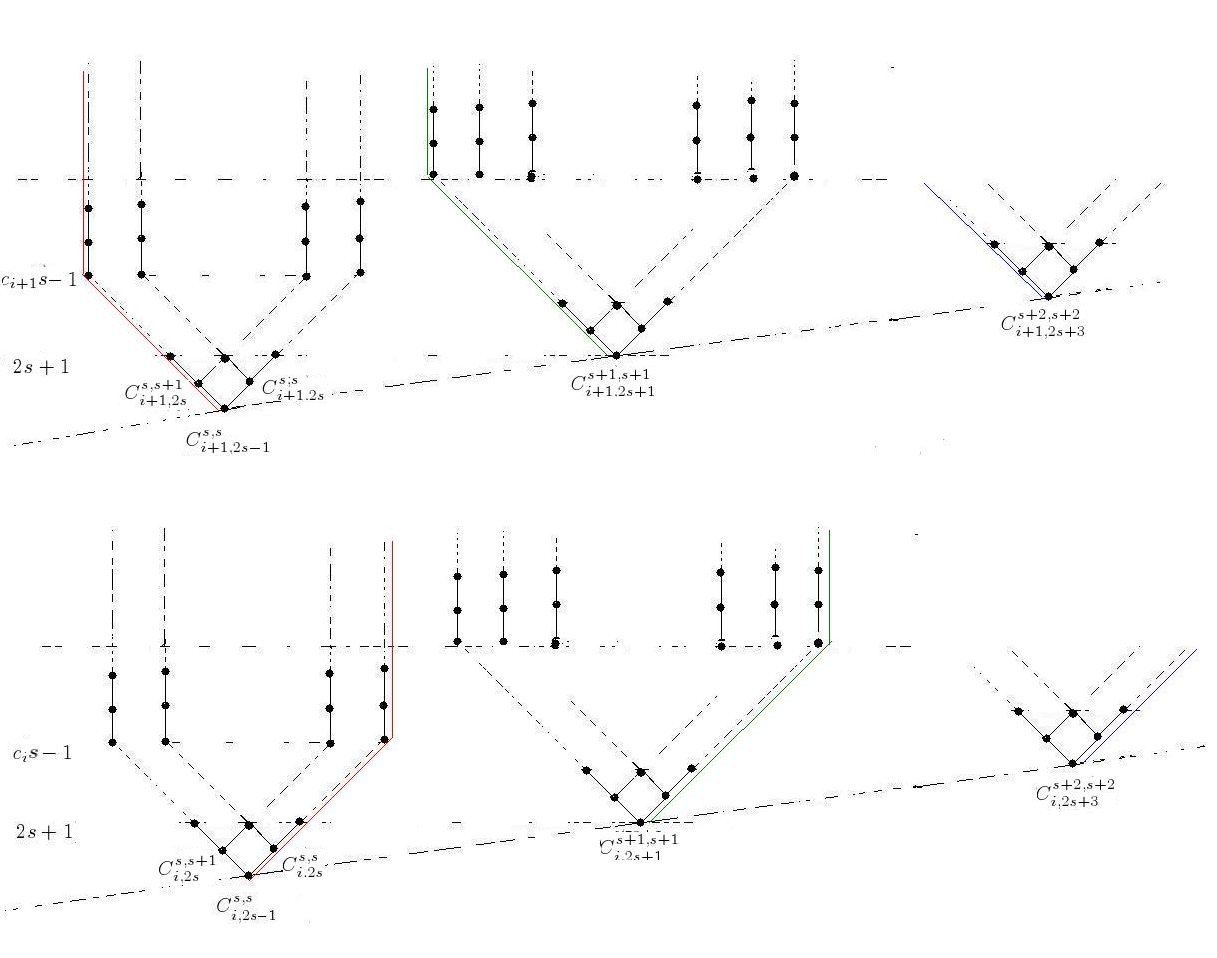}
\begin{defi}\label{sp}
 Let $m \in \IN,~m\geq 1,$ and let $C$ be an irreducible component of $S_m^0.$ By Theorem \ref{th}, there exist  
$ s \in \{1,\ldots,\lceil\frac{m}{2}\rceil\}, $ 
$ l \in \{s,\ldots,m_i^s \} $ and  $i \in \{ 2,\cdots, e-1\}$ such that $C=C_{i,m}^{s,l}.$
We say that $C$ has index of speciality $s.$ Note that $s=ord_{\gamma}(M):=\textnormal{min}_{f\in M}\{ord_{\gamma}(f)\}$ where $M$ 
is the maximal ideal of the local ring $O_{S,0}$ and $\gamma$ the generic point of $C.$ 
\end{defi}
For $a,b\in \mathbb{N},~b\not=0,$ we denote by
$[\frac{a}{b}]$ the integral part of $\frac{a}{b}.$ 
 For $c,m \in \mathbb{N}$, let $q_c=[\frac{m}{c}].$ We set
$$N_c^{s}(m):=(sc-(2s-1)), for~~s=1,...,q_c~;$$ 
$$N_c^{s}(m):=m-(2s-2),for~~s=q_c+1,...,\lceil\frac{m}{2}\rceil.$$
 For $m\in \mathbb{N},~m\geq 1$, we call $N(m)$ the number of irreducible component of $S_m^0.$ Then counting the irreducible components in the Theorem \ref{th} we find
\begin{cor}\label{nb}
 If all the $c_i$ are equal to $2,$ then $N(m)=\lceil\frac{m}{2}\rceil.$ Otherwise let $c_{i_1},...,c_{i_h}$ be the elements
 in $\{c_2,\ldots,c_{e-2}\}$ different from $2,$ 
then we have $$N(m)=\sum_{s=1}^{\lceil\frac{m}{2}\rceil}(N^s_{c_{i_1}}(m)+(N^s_{c_{i_2}}(m)-1)+\ldots+(N^s_{c_{i_h}}(m)-1)).$$
\end{cor}

\begin{cor} \label{Ni}
Let $S$ be a toric surface. The number of irreducible components of $S_m^0$ and their dimensions determine the set 
$\{ c_t,t=2,\ldots,e-2\}.$
\end{cor}
\begin{dem}
 We have that $\dim(S_1^0)=e,$ the embedding dimension of $S.$ If $e=3,$ then for $m$ big enough, we have by theorem \ref{lc}
that $N(m)=c$ is constant, and we deduce that $S$ is an $A_c$ singularity. Suppose the $e\geq 4.$\\
For $m\geq 1,$ let $$\tilde{N}_1(m)
=\sum_{s=1}^{\lceil\frac{m}{2}\rceil}((m+1-(2s-1))+(e-3)(m+1-(2s-1)-1).$$
We have that $N(m)\leq\tilde{N}_1(m) $ and $N(1)=\tilde{N}_1(1)=1.$
Let $$m_1=min\{m~;~ N(m)<\tilde{N}_1(m) \}~~ \mbox{and}~~ \alpha_1=\tilde{N}_1(m_1)-N(m_1),$$
then there exists $i_1,\cdots,i_{\alpha_1}\in\{c_2,\ldots,c_{e-1}\}$ such that $c_{i_1}=\cdots=c_{i_{\alpha_1}}=m_1.$\\
If $\alpha_1=e-2,$ then we have found all the $c_i.$ If not, then for $j\geq 2,$ we recursively define 
$$\tilde{N}_{j}(m)=\sum_{s=1}^{\lceil\frac{m}{2}\rceil}( N^s_{c_{i_1}}(m)+ (N^s_{c_{i_2}}(m)-1)+\cdots+ (N^s_{c_{i_{\alpha_1}}}(m)-1)
+\cdots+$$
$$(N^s_{c_{i_{\alpha_1+\cdots+\alpha_{j-1}}}}(m)-1))+
(e-2-(\alpha_1+\cdots+\alpha_{j-1}))(m+1-(2s-1)-1),$$

$$m_j=min\{m~;~ N(m)<\tilde{N}_j(m) \}~~ \mbox{and}~~ \alpha_j=\tilde{N}_j(m_j)-N(m_j).$$
Therefore there exists 
$i_{\alpha_1+\cdots+\alpha_{j-1}+1},\cdots,i_{\alpha_1+\cdots+\alpha_{j-1}+\alpha_j}\in\{c_2,\ldots,c_{e-1}\}$ such that
 $$c_{i_{\alpha_1+\cdots+\alpha_{j-1}+1}}=\cdots=c_{i_{\alpha_1+\cdots+\alpha_{j-1}+\alpha_j}}=m_j.$$
If $\alpha_1+\cdots+\alpha_{j-1}+\alpha_j=e-2,$ then we have found all the $c_t,$ otherwise 
we repeat the procedure at most $e-2$ times.
\end{dem}

\begin{rem}
 Corollary $\ref{Ni}$ is to compare with the result of Nicaise in \cite{Ni}, where he proved that the motivic Igusa Poincaré series
of a toric surface is equivalent to the set $\{ c_t,t=2,\ldots,e-2\},$ and that the order of the $c_i$ in the continued fraction can not 
be extracted from this series. It is clear also from the formulas given in proposition \ref{irr} and corollary \ref{nb}, that
 the number of irreducible components and their dimensions is not affected by the order of the $c_i$ in the continued fraction.
 Note that despite that these informations on the jet schemes are closely related to the informations encoded in the motivic Igusa Poincaré series,
 they are not equivalent in general.
Below we show how  we extract all the $c_i$ or equivalently the analytical type of $S$ from their jet schemes.      
\end{rem}
\begin{cor}\label{gra}
 Let $S$ be a toric suface. The weighted graph that we have associated to the irreducible components of $S_m^0$ is equivalent to the data of
all the $c_i$ and of their order in the continued fraction, or equivalently to the analytical type of $S.$ 
\end{cor}
\begin{rem}
Note that if we reverse  the order of the $c_t,$ the obtained toric surface will be isomorphic to the original one.
\end{rem}
\begin{dem}
 By corollary \ref{Ni}, We just need to show that we can extract the order of the $c_t$. Given an irreducible component $C$ of 
$S_2^{0},$ then if there exists a unic $i$ such that $C=C^{1,l}_{i,2}$ then $c_i$ is extremal in the continued fraction. If not, 
let $i_1$ and $i_2$ be such that $C=C^{1,l}_{i_1,2}=C^{1,l}_{i_2,2},$ then $c_{i_1}$ and $c_{i_2}$ are neihbours in the continued fraction, 
and the corollary follows. On the graph this can be seen on the broken lines that we indentify.
   
\end{dem}

Using a theorem of Mustata in \cite{Mus2}, we obtain as a byproduct the log canonical threshold $lct(S,\mathbb{A}^e)$ of the pair $S \subset \mathbb{A}^e:$
\begin{cor}
 Let $S$ be a toric surface of embedding dimension $e.$ If $e=3$ (i.e. $S$ is an $A_n$ singularity) then  $lct(S,\mathbb{A}^e)=1,$ otherwise
$$lct(S,\mathbb{A}^e)=\frac{e}{2} $$
\end{cor}
\begin{dem}
 By \cite{Mus2} we have that $$lct(S,\mathbb{A}^e)=\min_{m \in \mathbb{N}}\frac{Codim(S_m,\mathbb{A}^e_m)}{m+1}.$$
The case $e=3$ follows from section 3, since in this case we have that $S_m$ is irreducible of codimension $m+1.$
Let us suppose that $e\geq 4.$ If $m$ is odd, $m=2s-1,~s\geq 1$ then the  component $C_{i,2s-1}^{s,s}$ is of maximal dimension and we have that  
$$\frac{Codim(C_{i,2s-1}^{s,s},\mathbb{A}^e_{2s-1})}{2s}=\frac{se}{2s}=\frac{e}{2}.$$ 
If $m$ is even, $m=2n,~n\geq0$ then the components $C_{i,2n}^{n,l},~i=2,\ldots,e-1,~l=n,m_i^n$  are of maximal dimension, and since $e\geq 4$ we have that 
 $$\frac{Codim(C_{i,2n}^{n,l},\mathbb{A}^e_{2n})}{2n+1}=\frac{ne+e-2}{2n+1}\geq \frac{e}{2},$$
and the lemma follows.

\end{dem}

\begin{cor}\label{Na} For $m\geq \mbox{max}\{c_i,~i=2,\cdots,e-1\},$ the number of irreducible components of $S_m^0,$ with index of speciality 
 $s=1,$ is equal to the number of exceptional divisors that appear on the minimal resolution of $S.$
\end{cor}
\begin{proof}
This comes from the comparaison of corollary \ref{nb} with proposition \ref{oda}.
 \end{proof}
\begin{rem}
 The corollary \ref{Na} is to compare with the bijectivity of the Nash map, due to Ishii and Kollar for this type of Singularities, \cite{IK}.
\end{rem}

\thanks{Laboratoire de Math\'ematiques de Versailles,
 Universit\'e de Versailles-St-Quentin-en-Yvelines, 45 avenue des \'Etats-Unis,
78035 Versailles Cedex, France.}\\
\thanks{ Email address: mourtada@math.uvsq.fr}
 

\begin{thebibliography}{0}
\bibitem[B-GS]{B-GS} C. Bouvier, G. Gonzalez-Sprinberg,  Système générateur minimal, diviseurs essentiels et $G$-désingularisations de variétés toriques, Tohoku Math. J. (2)  47  (1995),  no. 1, 125-149.
\bibitem[BMS]{BMS} C. Bruschek, H. Mourtada, J. Schepers, Arc spaces and Rogers-Ramanujan identities, Preprint 2011. 
\bibitem[dFEI]{dFEI} T. de Fernex, L. Ein, S. Ishii, Divisorial valuations via arcs.  Publ. Res. Inst. Math. Sci.  44  (2008),  no. 2, 425-448, 
\bibitem[DL1]{DL1} J. Denef, F. Loeser,  Germs of arcs on singular algebraic varieties and motivic integration,  Invent. Math.  135 (1999),  no. 1, 201-232.
\bibitem[DL2]{DL2} J. Denef, F. Loeser, On some rational generating series occurring in arithmetic geometry.  
Geometric aspects of Dwork theory. 2004.
\bibitem[D]{D} R.Docampo, Arcs on determinantal varieties, preprint(2010).
\bibitem[ELM]{ELM} L. Ein, R. Lazarsfeld, M. Mustata, Contact loci in arc spaces.  Compos. Math.  140  (2004),  no. 5, 1229-1244.
\bibitem[EM]{EM} L. Ein, M. Mustata, Jet schemes and singularities, Proc. Sympos. Pure Math., 80, Part 2, Amer. Math. Soc., Providence, RI (2009) 505-546.
\bibitem[GS]{GS} R. Goward, K. Smith, The jet scheme of a monomial scheme, Comm. Algebra  34  (2006),  no. 5, 1591-1598.
\bibitem[I]{I} S. Ishii, The arc space of a toric variety.  J. Algebra  278  (2004),  no. 2, 666-683. 
\bibitem[IK]{IK} S. Ishii, J. Kollár, The Nash problem on arc families of singularities. Duke Math. J.  120  (2003),  no. 3, 601-620. 
\bibitem[KKMS]{Mum}G. Kempf, F.F. Knudsen, D. Mumford, B. Saint-Donat, 
Toroidal embeddings I,Lecture Notes in Mathematics, Vol. 339. Springer-Verlag, Berlin-New York, 1973.
\bibitem[K]{K} M. Kontsevich, Lecture at Orsay, 1995.
\bibitem[Ko]{Ko} A. G. Kouchnirenko, Poly\`edres de Newton et nombres de Milnor, Invent. Math., 32, 1--31, 1976.
\bibitem[L]{L} M. Lejeune-Jalabert, Arcs analytiques et résolution minimale des singularités des surfaces quasi-homogènes,
Séminaire sur les Singularités des Surfaces, Lecture Notes in Mathematics, 777. Springer, Berlin, (1980) 304-336. 
\bibitem[LR]{LR} M. Lejeune-Jalabert, A. Reguera, J. The Denef-Loeser series for toric surface singularities. Proceedings of the International Conference on Algebraic Geometry and Singularities 
(Spanish) (Sevilla, 2001).  Rev. Mat. Iberoamericana  19  (2003),  no. 2, 581-612.
\bibitem[Mo1]{Mo1} H. Mourtada, Jet schemes of complex branches and equisingularity, To appear in Annales de l'Institut Fourier.
\bibitem[Mo2]{Mo2} H. Mourtada, Jet schemes of rational double point singularities, Preprint 2010.
\bibitem[Mo3]{Mo3} H. Mourtada, Jet schemes of toric surfaces, to appear in Comptes Rendus de l'Académie des Sciences.
\bibitem[Mus1]{Mus1} M. Mustata, with an appendix by David Eisenbud and Edward Frenkel, Jet schemes of locally complete intersection canonical singularities,  Invent. Math.  145  (2001),  no. 3, 397-424.
\bibitem[Mus2]{Mus2} M. Mustata, Singularities of pairs via jet schemes.  J. Amer. Math. Soc.  15  (2002),  no. 3, 599-615. 
\bibitem[N]{N}J. F. Nash, Arc structure of singularities, Duke Math. J. 81 , (1995) no. 1, 31-38 .
\bibitem[Ni]{Ni} J. Nicaise, Motivic generating series for toric surface singularities.  Math. Proc. Cambridge Philos. Soc.  138  (2005), 
 no. 3, 383-400.
\bibitem[O]{O} T. Oda, Convex bodies and algebraic geometry. An introduction to the theory of toric varieties, (3) [Results in Mathematics and Related Areas (3)], 15. Springer-Verlag, Berlin, 1988.
\bibitem[R]{R} O. Riemenschneider, Zweidimensionale Quotientensingularitäten: Gleichungen und Syzygien, Arch. Math. (Basel)  37  (1981), no. 5, 406-417.
\bibitem[St]{St}J. Stevens, Deformations of singularities. Lecture Notes in Mathematics, 1811. Springer-Verlag, Berlin, 2003.
\bibitem[V]{V}P. Vojta, Jets via Hasse-Schmidt derivations.  Diophantine geometry,  335-361, CRM Series, 4, Ed. Norm., Pisa, 2007.
\end{thebibliography}
\end{document}